\documentclass{amsart}

\usepackage{fancyhdr}

\usepackage{hyperref}
\usepackage{enumerate}
\usepackage{amssymb}

\usepackage{url}

\usepackage{accents}

\newtheorem{theorem}{Theorem}[section]
\newtheorem{lemma}[theorem]{Lemma}
\newtheorem{corollary}[theorem]{Corollary}

\theoremstyle{definition}
\newtheorem{definition}[theorem]{Definition}

\theoremstyle{remark}

\numberwithin{equation}{section}

\providecommand{\abs}[1]{\left\lvert{#1}\right\rvert}

\providecommand{\norm}[1]{\left\lVert{#1}\right\rVert}

\newcommand{\dual}[2]{\left\langle{#1}, \, {#2}\right\rangle}

\newcommand{\lr}[1]{\left({#1}\right)}
\newcommand{\lra}[1]{\left\lbrack{#1}\right\rbrack}
\newcommand{\lrb}[1]{\left\lbrace{#1}\right\rbrace}

\newcommand{\calka}[2][x]{\int_{\Omega} {#2} \, \mathrm{d}{#1}}

\newcommand{\calkawoddoz}[4][s]{\int_{#2}^{#3} {#4} \, \mathrm{d}{#1}}
\newcommand{\calkat}[2][t]{\int_{0}^{T} {#2} \, \mathrm{d}{#1}}

\newcommand{\dv}{\operatorname{div}}
\newcommand{\supp}{\mathop{\rm supp}}

\newcommand{\T}{{\ensuremath{\mathbb T}}}
\newcommand{\Z}{{\ensuremath{\mathbb Z}}}
\newcommand{\PP}{\ensuremath{\mathbb{P}}}

\def\ri{{\mathrm{i}}}

\begin{document}

\title{Robustness of Regularity for the $3$D Convective Brinkman--Forchheimer Equations}

\author{Karol W. Hajduk}
\address{Department of Mathematics and Statistics, Faculty of Science, Masaryk University,
Building 08, Kotl\'a{\v r}sk\'a 2, 611 37, Brno, Czech Republic}
\email{hajduk@math.muni.cz}
\thanks{KH was supported by an EPSRC Standard DTG EP/M506679/1 and by the Warwick Mathematics Institute.}

\author{James C. Robinson}
\address{Mathematics Institute, Zeeman Building, University of Warwick,
Coventry, CV4 7AL, United Kingdom}
\email{j.c.robinson@warwick.ac.uk}
\thanks{JR was supported in part by an EPSRC Leadership Fellowship EP/G007470/1.}

\author{Witold Sadowski}
\address{School of Mathematics, University of Bristol,
University Walk, Clifton, Bristol, BS8 1TW, United Kingdom}
\email{w.sadowski@bristol.ac.uk}


\date{January 29, 2021.}


\begin{abstract}
We prove a robustness of regularity result for the $3$D convective Brinkman--Forchheimer equations
$$ \partial_tu -\mu\Delta u + (u \cdot \nabla)u + \nabla p + \alpha u + \beta\abs{u}^{r - 1}u = f, $$
for the range of the absorption exponent $r \in [1, 3]$ (for $r > 3$ there exist global-in-time regular solutions), i.e.\ we show that strong solutions of these equations remain strong under small enough changes of the initial condition and forcing function. We provide a smallness condition which is similar to the robustness conditions given for the $3$D incompressible Navier--Stokes equations by Chernyshenko et al.\ \cite{CCRT} and Dashti \& Robinson \cite{DashtiRobinson}.
\end{abstract}

\maketitle

\section{Introduction}

\makeatletter
    \renewcommand{\theequation}{{\thesection}.\@arabic\c@equation}
    \renewcommand{\thetheorem}{{\thesection}.\@arabic\c@theorem}
    \renewcommand{\thedefinition}{{\thesection}.\@arabic\c@definition}
    \renewcommand{\thecorollary}{{\thesection}.\@arabic\c@corollary}
    \renewcommand{\theproposition}{{\thesection}.\@arabic\c@proposition}
\makeatother

In this paper we consider strong solutions of the $3$D incompressible convective Brinkman--Forchheimer equations (CBF, see e.g.\ \cite{YouZhao})
\begin{equation} \label{cbfr}
\partial_tu -\mu\Delta u + (u \cdot \nabla)u + \nabla p + \alpha u + \beta\abs{u}^{r - 1}u = f, \quad \nabla \cdot u = 0,
\end{equation}
where $u(t, x) = (u_1, u_2, u_3)$ is the velocity field, the scalar function $p(t, x)$ is the pressure and $f(t, x) = (f_1, f_2, f_3)$ are given body forces acting within the fluid. The constant $\mu$ denotes the positive Brinkman coefficient (effective viscosity). The positive constants $\alpha$ and $\beta$ denote respectively the Darcy (permeability of porous medium) and Forchheimer (porosity of the material) coefficients. The `absorption exponent' $r$ can be greater or equal than $1$. The equations $(\ref{cbfr})$ can be seen also as the~Navier--Stokes equations (NSE) modified by an absorption term $\abs{u}^{r - 1}u$ (e.g.\ as in \cite{Oliveira}) or as the tamed Navier--Stokes equations (e.g.\ as in \cite{RocknerZhang}). Although the motivation of adding the~absorption term to the NSE is rather of mathematical nature, there are some relevant physical justifications and applications of this model (for more details see, for example, \cite{HR}, \cite{KalantarovZelik}, \cite{MarkowichTiti}, and references therein).

For simplicity, we often assume that $\mu, \alpha, \beta = 1$, but all of these coefficients can be taken as arbitrary nonnegative constants. In our arguments we omit the~linear term $\alpha u$ in $(\ref{cbfr})$ since its treatment does not cause any additional mathematical difficulties. Furthermore, we also neglect frequently the external forces ($f \equiv 0$) and consider the unforced CBF equations
\begin{equation} \label{cbfralphazero}
\partial_tu - \Delta u + (u \cdot \nabla)u + \nabla p + \abs{u}^{r - 1}u = 0.
\end{equation}

It was first established in \cite{KalantarovZelik} that, when the absorption exponent $r$ is greater than $3$, there exist global-in-time regular solutions of the CBF equations on bounded domains with smooth boundary (with the Dirichlet boundary condition). The same result was proved in \cite{HR} for the periodic case (with $r > 3$), and also for the critical exponent ($r = 3$) when the product of viscosity and porosity coefficients is not too small, i.e.\ $4\mu\beta \geq 1$.

In this paper we consider the equations $(\ref{cbfr})$ on a three-dimensional torus $\mathbb{T}^3$ with periodic boundary conditions. In this setting it is often convenient to assume zero mean-value constraint for the functions (i.e.\ $\int{u(t, x)\, \mathrm{d}x} = 0$). However, we cannot do that for the CBF equations $(\ref{cbfr})$ because the absorption term $\abs{u}^{r - 1}u$ does not preserve this property. Therefore, we cannot use the usual Poincar\'e inequality $\norm{u}_{L^2} \le c\norm{\nabla u}_{L^2}$ and we have to control the full $H^1$-norm instead.

We define the Sobolev spaces $H^s(\T^3)$ for $s \ge 0$ by the Fourier expansion
$$ H^s(\T^3) := \lrb{u \in L^2(\T^3) : u(x) = \sum_{k \in \Z^3}{\hat{u}_k}e^{\ri k \cdot x}, \quad \hat{u}_k = \overline{\hat{u}_{-k}}, \quad \norm{u}_{H^s(\T^3)} < \infty }, $$
where
$$ \norm{u}_{H^s(\T^3)}^2 := \abs{\T^3}\sum_{k \in \Z^3}{ (1 + \abs{k}^{2s})\abs{\hat{u}_k}^2 } $$
and
$$ \hat{u}_k := \abs{\T^3}^{-1}\int_{\T^3}{u(x)e^{-\ri k \cdot x}\, \mathrm{d}x}. $$

We will also use the following function spaces:
\begin{align*}
\mathcal{D}_{\sigma} &:= \lrb{\varphi \in [C^\infty(\T^3)]^3 : \supp\varphi \mbox{ is compact}, \nabla \cdot \varphi = 0}, \\
L^p_{\sigma} &:= \mbox{closure of } \mathcal{D}_{\sigma} \mbox{ in the } L^p\mbox{-norm} \quad \mbox{for} \quad p \geq 1, \\
V^s &:= \mbox{closure of } \mathcal{D}_{\sigma} \mbox{ in the } H^s\mbox{-norm} \quad \mbox{for} \quad s \geq 1.
\end{align*}
We denote the Hilbert space $L^2_{\sigma}$, which is of great importance in the theory of fluid mechanics, by $H$. This space is endowed with the inner product induced by $L^2(\T^3)$. We denote it by $\dual{\cdot}{\cdot}$ and the~corresponding norm is denoted by $\norm{\cdot}$. When we consider the CBF equations with forcing, we assume that $f$ belongs at least to the space $L^1(0, T; H)$.

The idea of `robustness of regularity' was first introduced by Constantin \cite{Constantin}, where it was shown that, under some conditions, regular solutions of the Euler equations are also regular solutions of the Navier--Stokes equations with small viscosity. This idea was further developed by Chernyshenko et al.\ \cite{CCRT}, Dashti \& Robinson \cite{DashtiRobinson} and by Marin-Rubio et al.\ \cite{MRRS} (for bounded sets of initial data) purely for the Navier--Stokes equations. It was also successfully applied by Bl\"omker et al.\ \cite{BlomkerNoldeRobinson} for a $1$D surface growth model that has striking similarities to the NSE. Similar ideas of propagation of regularity for the Navier--Stokes equations can be found in the book by Chemin et al.\ \cite{CheminGeo}. In the~ present paper we show that the robustness of regularity of strong solutions holds for the convective Brinkman--Forchheimer equations (with $r \in [1, 3]$) as well.

We note that a similar result to our `robustness of regularity' can be achieved via the Implicit Function Theorem. However, this method does not provide the explicit bounds for the closeness of considered solutions, which our method does. These bounds can potentially be used for numerical verification of regularity for the CBF equations, as it was done for the Navier--Stokes equations in \cite{CCRT, DashtiRobinson} and \cite{MRRS}. Similar ideas have been used in the context of singularly perturbed, damped wave equation with supercritical nonlinearities in \cite{Zelikrobustness} (cf.\ Lemma $2.1$ and Proposition $4.1$), to show regularity of the global attractor. In more recent work, \cite{Zeliketalrobustness}, these type of methods have been used to show a convergence of the family of attractors of the relaxed hyperbolic $2$D Navier--Stokes equations in bounded domains, to the attractor of the non-relaxed problem.

The paper is organised as follows: in Section $\ref{prelim}$ we introduce some technical tools which will be crucial in dealing with the nonlinear term $\abs{u}^{r - 1}u$. Section $\ref{local}$ is devoted to the local existence of strong solutions which is needed in the proof of the robustness of regularity result. In Section $\ref{uniqueness}$ we provide general properties of strong solutions of the CBF equations with the absorption exponent $r$ in the range $[1, 3]$. We also establish the uniqueness of strong solutions in the larger class of weak solutions satisfying the energy inequality, i.e.\ a `weak-strong uniqueness' property. In the last sections, we prove the main result of the~paper (Section $\ref{s:robustness}$) and discuss its possible applications to the convective Brinkman--Forchheimer equations (Section $\ref{s:conclusion}$).

In our estimates, we frequently use a constant $c >0$, whose value can change from line to line.

\section{Preliminaries} \label{prelim}

\makeatletter
    \renewcommand{\theequation}{{\thesection}.\@arabic\c@equation}
    \renewcommand{\thetheorem}{{\thesection}.\@arabic\c@theorem}
    \renewcommand{\thedefinition}{{\thesection}.\@arabic\c@definition}
    \renewcommand{\thecorollary}{{\thesection}.\@arabic\c@corollary}
    \renewcommand{\theproposition}{{\thesection}.\@arabic\c@proposition}
\makeatother

For notational convenience we denote the terms connected with the additional nonlinearity in the convective Brinkman--Forchheimer equations by $C_r$. For $r > 0$ and for all functions $u, v \in L^{r + 1}_{\sigma}$ we define
$$ C_r(u, v) := \mathbb{P}\lr{\abs{u}^{r - 1}v}, $$
where $\mathbb{P} : L^p \to L^p_{\sigma}$ is the Leray projector in $L^p$ (see e.g.\ \cite{FujiwaraMorimoto} for details); additionally we define
$$ C_r(u) := C_r(u, u). $$

We have the following crucial properties of the nonlinearity $C_r$.

\begin{lemma}
\label{monotone}
For every $r \ge 1$ and for all functions $u, v \in L^{r + 1}_{\sigma}$ we have a lower bound
\begin{equation} \label{lowerbound}
\dual{C_r(u) - C_r(v)}{u - v} = \dual{\abs{u}^{r - 1}u - \abs{v}^{r - 1}v}{u - v} \geq c\norm{u - v}_{r + 1}^{r + 1},
\end{equation}
where $c$ is a positive constant depending only on $r$, and $\dual{\cdot}{\cdot}$ is the inner product in $L^2$.
\end{lemma}

It immediately follows from $(\ref{lowerbound})$ that for $r \ge 1$ the nonlinearity $C_r$ is monotone in the sense that
\begin{equation} \label{monotonicity}
\dual{C_r(u) - C_r(v)}{u - v} \ge 0
\end{equation}
for all $u, v \in L^{r + 1}_{\sigma}(\T^3)$. One can show $(\ref{monotonicity})$ independently even for $r > 0$ by direct computation and using only Young's inequality.

Lemma $\ref{monotone}$ is a consequence of properties of vectors $\abs{u}^{r - 1}u$ in $\mathbb{R}^n$ $(n \ge 1)$.
The proof of the lower bound $(\ref{lowerbound})$ is taken from \cite{DiBenedetto} with some minor changes.
\begin{proof}
For all $u, v \in \mathbb{R}^n$ we observe that
\begin{align} \nonumber
\lr{\abs{u}^{r - 1}u - \abs{v}^{r - 1}v} = {\int_0^1{\frac{\mathrm{d}}{\mathrm{d}s}\lr{\abs{su + (1 - s)v}^{r - 1}(su + (1 - s)v)}\, \mathrm{d}s}}
\end{align}
and hence
\begin{align} \nonumber
\lr{\abs{u}^{r - 1}u - \abs{v}^{r - 1}v} &\cdot w = {\int_0^1{ {\abs{su + (1 - s)v}^{r - 1}}\abs{w}^2\, \mathrm{d}s} } \\ \nonumber
&+ (r - 1)\int_0^1{ \abs{su + (1 - s)v}^{r - 3}\lr{[su + (1 - s)v]\cdot w}^2 \, \mathrm{d}s},
\end{align}
where $w := u - v$. Therefore, we obtain for $r \ge 1$
$$ \lr{\abs{u}^{r - 1}u - \abs{v}^{r - 1}v} \cdot w \ge \abs{w}^2{\int_0^1{ {\abs{su + (1 - s)v}^{r - 1}}\, \mathrm{d}s} }. $$

If $\abs{u} \ge \abs{v - u}$, we have
$$ \abs{su + (1 - s)v} \ge \abs{\abs{u} - (1 - s)\abs{w}} \ge s\abs{w} $$
and we can conclude that
$$ \lr{\abs{u}^{r - 1}u - \abs{v}^{r - 1}v} \cdot w \ge \frac{1}{r}\abs{w}^{r + 1}. $$

On the other hand, if $\abs{u} < \abs{v - u}$, we have
\begin{align} \nonumber
\abs{w}^2{\int_0^1{ {\abs{su + (1 - s)v}^{r - 1}}\, \mathrm{d}s} } &\ge \abs{w}^2{\int_0^1{ \frac{ (\abs{su + (1 - s)v}^2)^{(r + 1)/2} }{(2 - s)^2\abs{w}^2}\, \mathrm{d}s} } \\ \nonumber
&\ge \frac{1}{4}\lr{\int_0^1{ {\abs{su + (1 - s)v}^2}\, \mathrm{d}s} }^{(r + 1)/2} \\ \nonumber
&= \frac{1}{4\cdot3^{(r + 1)/2}}\lr{\abs{u}^2 + u \cdot v + \abs{v}^2}^{(r + 1)/2} \\ \nonumber
&\ge c\abs{w}^{r + 1}.
\end{align}

Finally, we observe an equality for $u, v \in L^{r + 1}_{\sigma}(\T^3)$
\begin{align} \nonumber
\dual{C_r(u) - C_r(v)}{w} &= \dual{\mathbb{P}\lr{\abs{u}^{r - 1}u - \abs{v}^{r - 1}v}}{w} \\ \nonumber
&= \dual{\abs{u}^{r - 1}u - \abs{v}^{r - 1}v}{w} \\ \nonumber
&= \int_{\T^3}{\lr{\abs{u}^{r - 1}u - \abs{v}^{r - 1}v}\cdot w \, \mathrm{d}x},
\end{align}
which ends the proof of the lemma due to monotonicity of integral and the above vector estimates.
\end{proof}

In what follows it will be essential to bound the difference
\begin{align} \label{difference}
\abs{u}^{r - 1}u - \abs{v}^{r - 1}v
\end{align}
in terms of only $u$ and $w$, where $w := u - v$.

\begin{lemma} \label{cbfdifference}
Let $u, v \in \mathbb{R}^n$. Then for $r \geq 1$
\begin{align} \nonumber
\abs{\abs{u}^{r - 1}u - \abs{v}^{r - 1}v} \leq (2^{r - 2}r)\lr{\abs{u}^{r - 1}\abs{w} + \abs{w}^{r}}.
\end{align}
\end{lemma}

\begin{proof}
First, we consider  the following function of one real variable $\varphi : \mathbb{R} \to \mathbb{R}^n$
\begin{equation} \nonumber
\varphi(\lambda) := \abs{u -\lambda w}^{r - 1}(u - \lambda w),
\end{equation}
for $\lambda \in \lra{0, 1}$.

It is easy to see that
$$ \varphi(1) -\varphi(0) = -\lr{\abs{u}^{r - 1}u - \abs{v}^{r - 1}v}, $$
and that the derivative of $\varphi$ equals
\begin{equation} \nonumber
\varphi'(\lambda) = -r\abs{u - \lambda w}^{r - 1}w.
\end{equation}

By the mean value theorem we estimate the difference $\lr{\ref{difference}}$
\begin{align} \nonumber
\abs{\abs{u}^{r - 1}u - \abs{v}^{r - 1}v} &= \abs{\varphi(1) -\varphi(0)} \leq \max_{\lambda \in \lra{0, 1}} \abs{\varphi'(\lambda)} \\ \nonumber
&= \max_{\lambda \in \lra{0, 1}} \abs{-r\abs{u - \lambda w}^{r - 1}w} \leq r\abs{w}\max_{\lambda \in \lra{0, 1}} \abs{u - \lambda w}^{r - 1} \\ \nonumber
&\leq r\abs{w}\lr{\abs{u} + \abs{w}}^{r - 1} \leq r\abs{w}\lra{2^{r - 2}\lr{\abs{u}^{r - 1} + \abs{w}^{r -1}}} \\ \nonumber
&\leq (2^{r - 2}r)\lr{\abs{u}^{r - 1}\abs{w} + \abs{w}^{r}}.
\end{align}

We used here the following simple fact
$$ f(x) := \frac{(1 + x)^s}{1 + x^s} \leq f(1) = 2^{s - 1} $$
which holds for all $s \ge 0$.
\end{proof}

We will also make use of the following lemma, whose proof consists of integration by parts and differentiation of the absolute value function (see \cite{RobinsonSadowski} for the proof in the periodic case or \cite{DaVeiga} in the whole space).
\begin{lemma} \label{witeklemma}
For every ${r \geq 1}$, if $u \in H^2(\Omega)$, where $\Omega$ is either the whole space $\mathbb{R}^3$ or the three-dimensional torus $\mathbb{T}^3$, then
\begin{equation} \nonumber
\int_{\Omega} -\Delta u \cdot \abs{u}^{r - 1}u \, \mathrm{d}x \geq \int_{\Omega} \abs{\nabla u}^{2}\abs{u}^{r -1} \, \mathrm{d}x.
\end{equation}
\end{lemma}

Explicitly, the left-hand side of the above equals (integrating by parts)
\begin{align} \nonumber
\int_{\Omega} -\Delta u \cdot \abs{u}^{r - 1}u \, \mathrm{d}x = \int_{\Omega} \abs{\nabla u}^{2}\abs{u}^{r - 1} \, \mathrm{d}x + \frac{(r - 1)}{4}\calka{\abs{u}^{r - 3}\abs{\nabla{\abs{u}^{2}}}^2}.
\end{align}

In particular, by Lemma \ref{witeklemma}, we can write for the absorption term $\abs{u}^{r - 1}u$ (for $r \geq 1$ and for a divergence-free function $u \in V^2$)
\begin{equation} \label{cbfnonlineargen}
\calka{\abs{\nabla u}^{2}\abs{u}^{r -1}} \leq \dual{A u}{C_r(u)} \leq r\calka{\abs{\nabla u}^{2}\abs{u}^{r -1}},
\end{equation}
where $A := -\mathbb{P}\Delta$ is the familiar Stokes operator. We recall the well know fact that the operators $\mathbb{P}$ and $\Delta$ commute on the domains $\mathbb{T}^3$ and $\mathbb{R}^3$ but not neccessarily on an open, bounded domain $\Omega \subset \mathbb{R}^3$ (see e.g.\ \cite{RRSbook} for examples). Therefore, we can freely use $\lr{\ref{cbfnonlineargen}}$ throughout this paper.

In the proof of main result of this paper (Theorem $\ref{crucialestimate}$) it will be crucial to control the $L^6$-norm of the gradient of a function $u$ by the $L^2$-norm of $Au$.
\begin{lemma} \label{crucialestimate}
Let $u \in D(A)$ on the torus $\T^3$. Then there exists a constant $c > 0$ independent of $u$ such that
$$ \norm{\nabla u}_{L^6(\T^3)} \le c\norm{Au}. $$
\end{lemma}

\begin{proof}
First, we apply the Sobolev embedding $H^1 \hookrightarrow L^6$
$$ \norm{\nabla u}_{L^6} \le c\norm{\nabla u}_{H^1} = c\lr{\norm{\nabla{u}}^2 + \norm{D^2u}^2}^{1/2}. $$

We can, either by direct computation or by the Poincar\'e inequality (noting that $\nabla u$ has zero mean-value for a periodic function $u$), verify that
$$ \norm{\nabla{u}} \le c\norm{D^2u}. $$
Therefore, we have the desired bound
$$ \norm{\nabla u}_{L^6}^2 \le c\norm{D^2u}^2 = c\sum_{m, n = 1}^3\sum_{k \in \Z^3}{k_m^2k_n^2\abs{\hat{u}_k}^2} = c\sum_{k \in \Z^3}{\abs{k}^4\abs{\hat{u}_k}^2} = c\norm{Au}^2. $$
\end{proof}

\subsection{Definition of strong solutions for the CBF equations}

Strong solutions of the convective Brinkman--Forchheimer equations have similar properties to those of the Navier--Stokes equations. Actually, the additional nonlinear term provides better regularity of solutions. Before explaining this in more detail, let us first define here the strong solutions for the CBF equations in a similar manner to the solutions of the NSE. To this end, we will use the following definition of a~weak solution (for the treatment of weak solutions for the NSE see e.g.\ \cite{Galdi} or \cite{RRSbook}).
\begin{definition} \label{weakcbfr}
We say that the function $u$ is \emph{a weak solution} on the time interval $[0, T)$ of the unforced convective Brinkman--Forchheimer equations $(\ref{cbfralphazero})$ with the~initial condition $u_0 \in H$, if
$$ u \in L^{\infty}({0, T; H}) \cap L^{r + 1}({0, T; L^{r + 1}_{\sigma}}) \cap L^2({0, T; V^1}) $$
and
\begin{align} \nonumber
-\int_{0}^{t}{\dual{u(s)}{\partial_t\varphi(s)}\, \mathrm{d}s} + \int_{0}^{t}{\dual{\nabla u(s)}{\nabla \varphi(s)}\, \mathrm{d}s} + \int_{0}^{t}{\dual{(u(s) \cdot \nabla)u(s)}{\varphi(s)}\, \mathrm{d}s} \\ \label{weakform}
+ \int_{0}^{t}{\dual{\abs{u(s)}^{r - 1}u(s)}{\varphi(s)}\, \mathrm{d}s} = -\dual{u(t)}{\varphi(t)} + \dual{u_0}{\varphi(0)}
\end{align}
for almost every $t \in (0, T)$ and for all divergence-free (only in space variables) test functions $\varphi \in \mathcal{D}_{\sigma}({[0, T) \times \T^3})$.
\end{definition}
A function $u$ is called \textit{a global weak solution} if it is a weak solution on $[0, T)$ for every $T > 0$.

\begin{definition}
\textit{A Leray--Hopf weak solution} of the convective Brinkman--\linebreak Forchheimer equations $(\ref{cbfralphazero})$ with the initial condition $u_0 \in H$ is a weak solution satisfying the following \emph{strong energy inequality}:
\begin{align} \label{energyineq}
\norm{u(t_1)}^2 + 2\int_{t_0}^{t_1}{\lr{\norm{\nabla u(s)}^2 + \norm{u(s)}_{L^{r + 1}}^{r + 1}} \, \mathrm{d}s} \leq \norm{u(t_0)}^2
\end{align}
for almost all initial times $t_0 \in [0, T)$, including zero, and all $t_1 \in (t_0, T)$.
\end{definition}

It is known that for every $u_0 \in H$ there exists at least one global Leray--Hopf weak solution of $(\ref{cbfralphazero})$ (see \cite{Oliveira} for the proof). For the critical absorption exponent $r = 3$ it is also known that all weak solutions satisfy $(\ref{energyineq})$ with equality (see \cite{HR} for the proof of that fact in the periodic case; the same proof but with the approximating sequence changed accordingly can be repeated also for bounded domains; for the details see the upcoming paper \cite{FHR}, which addresses different simultaneous approximation schemes in Lebesgue and Sobolev spaces on bounded domains).

Just as for the Navier--Stokes equations we define strong solutions of the convective Brinkman--Forchheimer equations as weak solutions with additional regularity.

\begin{definition} \label{cbfstrongsol}
We say that a vector field $u$ is a \emph{strong solution} of the convective Brinkman--Forchheimer equations $\lr{\ref{cbfralphazero}}$ if, for the initial condition $u_0 \in V^1$, it is a~weak solution and additionally it possesses higher regularity, i.e.\
$$ u \in L^{\infty}(0, T; V^1) \cap L^2(0, T; V^2). $$
\end{definition}
As mentioned above, strong solutions of the CBF equations are in fact even more regular than those of the NSE. As shown in \cite{HR}, the extra dissipative term guarantees that every strong solution $u$ belongs additionally to the spaces
$$ L^{r + 1}(0, T; L^{3(r + 1)}_{\sigma}) \quad \mbox{and} \quad L^{r + 1}(0, T; \mathcal{N}^{2/(r + 1), r + 1}), $$
where $\mathcal{N}^{s, p} = B^{s, p}_{\infty}$ is a Nikol'ski\u\i \ space.

\section{Local existence of strong solutions} \label{local}

In this section, we prove local-in-time existence of strong solutions for the~convective Brinkman--Forchheimer equations with $r \ge 1$.

\begin{theorem} \label{localcbf}
For every initial condition $u_0 \in V^1$ there exists a time $T > 0$ such that a Leray--Hopf weak solution $u$ starting from $u_0$ is a strong solution of the convective Brinkman--Forchheimer equations $\lr{\ref{cbfralphazero}}$ on the time interval $\left[0, T\right]$. Additionaly $u$ satisfies the bound
\begin{equation} \label{additionalbound}
\int_{0}^{T}{\lr{\int_{\T^3}{\abs{\nabla u(t)}^2\abs{u(t)}^{r - 1}\, \mathrm{d}x}}\, \mathrm{d}t} < \infty.
\end{equation}
\end{theorem}

Below we only show formal calculations, which can be made rigorous, for example, by the use of Galerkin approximations and passage to the limit.

\begin{proof}
We work with the equations $\lr{\ref{cbfralphazero}}$ in their functional form
\begin{equation} \nonumber
\partial_t u + Au + B(u) + C_r(u) = 0,
\end{equation}
where $B(u, v) := \mathbb{P}(u \cdot \nabla)v$ and $B(u) := B(u, u)$ for $u, v \in V^1$. Let $u$ be a global Leray--Hopf weak solution of $\lr{\ref{cbfralphazero}}$ starting from $u_0 \in V^1$.
Multiplying formally $\lr{\ref{cbfralphazero}}$ by $Au$ and integrating over the spatial domain, we obtain
\begin{equation} \label{szacujcbf}
\frac{1}{2}\frac{\mathrm{d}}{\mathrm{d}t}\norm{\nabla u}^2 + \norm{Au}^2 + \dual{C_r(u)}{Au} \leq \abs{\dual{B(u)}{Au}}.
\end{equation}

First, we estimate the convective term $\dual{B(u)}{Au}$ using H\"older's, Sobolev's and Young's inequalities and also Lemma $\ref{crucialestimate}$
\begin{align} \nonumber
\abs{\dual{B(u)}{Au}} &\leq \int_{\T^3}{\abs{u}\abs{\nabla u}\abs{Au}\, \mathrm{d}x} \leq \norm{u}_{L^6}\norm{\nabla u}_{L^3}\norm{Au} \\ \nonumber
&\leq c\norm{u}_{H^1}\norm{\nabla u}^{{1}/{2}}\norm{\nabla u}_{L^6}^{{1}/{2}}\norm{Au} \\ \label{szacuj}
&\leq c\norm{u}_{H^1}^{{3}/{2}}\norm{Au}^{{3}/{2}} \leq c\norm{u}_{H^1}^{6} + \frac{1}{2}\norm{Au}^{2}.
\end{align}

We recall that, by Lemma \ref{witeklemma}, we have the inequality $\lr{\ref{cbfnonlineargen}}$
$$ \dual{C_r(u)}{Au} \geq \int_{\T^3}{\abs{\nabla u}^2\abs{u}^{r - 1}\, \mathrm{d}x} \geq 0 \quad \mbox{for} \quad r \geq 1. $$

Using this fact and also the estimate $\lr{\ref{szacuj}}$, we obtain from $\lr{\ref{szacujcbf}}$ a differential inequality
\begin{equation} \nonumber
\frac{1}{2}\frac{\mathrm{d}}{\mathrm{d}t}\norm{\nabla u}^2 + \frac{1}{2}\norm{Au}^2 + \int_{\T^3}{\abs{\nabla u}^2\abs{u}^{r - 1}\, \mathrm{d}x} \leq c\norm{u}_{H^1}^6.
\end{equation}
Noting that $u$ satisfies also [see energy inequality $(\ref{energyineq})$]
\begin{equation} \nonumber
\frac{1}{2}\frac{\mathrm{d}}{\mathrm{d}t}\norm{ u}^2 + \norm{\nabla u}^2 + \norm{u}_{L^{r + 1}}^{r + 1} \le 0,
\end{equation}
we have (adding the above inequalities and dropping some terms on the left-hand side)
\begin{equation} \nonumber
\frac{1}{2}\frac{\mathrm{d}}{\mathrm{d}t}\norm{u}_{H^1}^2 + \frac{1}{2}\norm{Au}^2 + \int_{\T^3}{\abs{\nabla u}^2\abs{u}^{r - 1}\, \mathrm{d}x} \leq c\norm{u}_{H^1}^6.
\end{equation}

By setting $X(t) := \norm{u(t)}_{H^1}^2$, we rewrite the above in the form
\begin{equation} \label{nierownosca}
X' + \norm{Au}^2 + 2\int_{\T^3}{\abs{\nabla u}^2\abs{u}^{r - 1}\, \mathrm{d}x} \leq cX^3,
\end{equation}
from which we obtain a differential problem
\begin{equation} \label{nierownoscODE}
\left\{\begin{split}
X' &\leq cX^3, \\
X(0) &= \norm{u_0}_{H^1}^2.
\end{split}\right.
\end{equation}
We would obtain the same differential inequality by following the above procedure for the Navier--Stokes equations (details for the NSE case can be found for example in \cite{RRSbook}). We can conclude that $X$ is no greater than the solution of $\lr{\ref{nierownoscODE}}$ turned into a differential equation instead of differential inequality. The solution of this ODE blows up in finite time $\tilde{T} = [2cX(0)^2]^{-1}$.
Therefore, for $0 \leq t \leq \tilde{T}/{2}$
\begin{equation} \nonumber
\norm{u(t)}_{H^1}^2 = X(t) \leq \frac{\norm{u_0}_{H^1}^2}{\sqrt{1 - 2c\norm{u_0}_{H^1}^4t}} \leq c\norm{u_0}_{H^1}^2.
\end{equation}
Using this bound and integrating $\lr{\ref{nierownosca}}$ over the time interval $[0, t]$, we obtain
\begin{equation} \nonumber
\norm{u(t)}_{H^1}^2 + \int_{0}^{t}{\norm{Au(s)}^2 \, \mbox{ds}} + 2\int_{0}^{t}{\lr{\int_{\T^3}{\abs{\nabla u(s)}^2\abs{u(s)}^{r - 1}\, \mathrm{d}x}}\, \mathrm{d}s} < \infty.
\end{equation}

Hence, we can conclude the proof with $T: = \lr{4c\norm{u_0}_{H^1}^4}^{-1}$.
\end{proof}

Theorem \ref{localcbf} tells us that the time of existence of strong solutions of the unforced CBF equations $(\ref{cbfralphazero})$ can be bounded below in terms of the initial condition
$$ T \gtrsim \norm{u_0}_{H^1}^{-4}. $$
We recall that we have the same situation for strong solutions of the Navier--Stokes equations; however, for the CBF equations we get the additional bound $(\ref{additionalbound})$.

\section{Uniqueness of strong solutions} \label{uniqueness}

In this section, we prove the uniqueness of strong solutions of the convective Brinkman--Forchheimer equations for incompressible fluids. By \emph{uniqueness} we mean here uniqueness of strong solutions in the larger class of weak solutions satisfying the energy inequality, which is often called `weak-strong uniqueness'. Classical uniqueness of strong solutions follows from that result since every strong solution is by definition a `{more regular}' weak solution.

To achieve our goal we need to establish some properties of strong solutions of the CBF equations. We follow here proofs of analogous results for the $3$D NSE equations, which can be found in many places, e.g.\ in \cite{Galdi} or \cite{RRSbook}.

First, we show that due to Definition \ref{cbfstrongsol} all the terms in the CBF equations are well-defined $L^2$ functions in the space-time domain.
\begin{lemma} \label{stronglemazero}
Let $u$ be a strong solution of the convective Brinkman--Forchheimer equations with $r \in [1, 3]$. Then
\begin{equation} \nonumber
\partial_tu, \quad \Delta u, \quad (u \cdot \nabla)u \quad \mbox{and} \quad \abs{u}^{r-1}u
\end{equation}
are all elements of $L^2(0, T; L^2)$.
\end{lemma}

\begin{proof}
We only need to consider the absorption term since the other terms can be dealt with in a similar way as in the analogous result for the Navier--Stokes equations.
We show that $\abs{u}^{r - 1}u$ is square integrable in the space-time domain [equaivalently that $u \in L^{2r}(0, T; L^{2r})$] using nesting of $L^p(\T^3)$ spaces and the Sobolev embedding
\begin{align} \nonumber
\calkat{\norm{\abs{u(t)}^{r - 1}u(t)}^2} &\leq \calkat{\norm{u(t)}_{L^{2r}}^{2r}} \leq c\calkat{ \norm{u(t)}_{H^1}^{2r}} \\ \label{tworestimate}
&\leq c\norm{u}_{L^{\infty}(0, T; H^1)}^{2r} < \infty.
\end{align}
\end{proof}

[Note that we can extend Lemma $\ref{stronglemazero}$ up to $r \le 5$. Indeed, by interpolation ($r \ge 3$) and Agmon's inequality in $3$D we have
\begin{align*}
\norm{u}_{L^{2r}}^{2r} \le \norm{u}_{L^6}^{6}\norm{u}_{L^{\infty}}^{2r - 6} \le \norm{u}_{L^6}^{6}\norm{u}_{H^1}^{r - 3}\norm{u}_{H^2}^{r - 3} \le c\norm{u}_{H^1}^{r + 3}\norm{u}_{H^2}^{r - 3}
\end{align*}
and therefore we obtain
\begin{align*}
\calkat{\norm{u(t)}_{L^{2r}}^{2r}} \le c\norm{u}_{L^{\infty}(0, T; H^1)}^{r + 3}\lr{\calkat{\norm{u(t)}_{H^2}^{2}}}^{(r - 3)/2},
\end{align*}
which is bounded for any strong solution $u$.]

The next result states that for almost all times the Leray projection of the unforced CBF equations $(\ref{cbfralphazero})$ is equal to zero.
\begin{lemma} \label{stronglemaone}
Let $u$ be a strong solution of the convective Brinkman--Forchheimer equations with $r \in [1, 3]$. Then
\begin{equation} \label{r:stronglemaone}
\calkat{\dual{\partial_tu -\Delta u + (u \cdot \nabla)u + \abs{u}^{r - 1}u}{w}} = 0
\end{equation}
for all $w \in L^2({0, T; H})$.
\end{lemma}

Again, the proof follows the same lines as in the Navier--Stokes case. We omit it here completely since, due to Lemma $\ref{stronglemazero}$, there are no additional problems caused by the absorption term $\abs{u}^{r - 1}u$.

The last property which we will need to prove the main result of this section states that a strong solution of the CBF equations (actually any function with the same regularity as a strong solution) can be used as a test function in the weak formulation $(\ref{weakform})$.
\begin{lemma} \label{stronglemmawo}
Suppose that $v$ is a weak solution of the convective Brinkman--\linebreak Forchheimer equations with $r \in \lra{1, 3}$. If $u$ has the regularity of a strong solution of the CBF equations, that is
$$ u \in L^2({0, T; H^2 \cap V^1}) \cap L^{r + 1}(0, T; L^{r + 1}) \quad \& \quad \partial_tu \in L^2({0, T; L^2}), $$
then for all times $t \in [0, T]$
\begin{align} \nonumber
-\calkawoddoz{0}{t}{\dual{v}{\partial_t u}} &+ \calkawoddoz{0}{t}{\dual{\nabla v}{\nabla u}} + \calkawoddoz{0}{t}{\dual{(v \cdot \nabla)v}{u}}
\\ \nonumber
&+ \calkawoddoz{0}{t}{\dual{\abs{v}^{r - 1}v}{u}} = \dual{v(0)}{u(0)} - \dual{v(t)}{u(t)}.
\end{align}
\end{lemma}

\begin{proof}
For each $t \in [0, T]$ we take
$$ u_n(t, x) := (u \ast \varphi_n)(t, x) = \int_{\T^3}{u(t, y)\varphi_n(x - y)\, \mathrm{d}y}, $$
where $\varphi_n(x) = n^3\varphi(nx)$ is a family of standard mollifiers connected with a positive, smooth function $\varphi(x)$ with compact support in $\T^3$. We recall that the sequence $u_n$ converges in $H^2$ and also in $L^{r + 1}(\T^3)$ with
$$ \norm{u_n(t)}_{L^{r + 1}} \le c\norm{u(t)}_{L^{r + 1}} $$
for a.e.\ $t \in [0, T]$, which is the key ingredient in adapting the proof from the Navier--Stokes case. Mollifying $u_n$ in time (see \cite{HR} for details of a similar argument) we obtain a sequence of test functions such that
\begin{align} \label{aa}
u_n &\to u \quad &\mbox{in} \quad &L^2(0, T; H^2), \\ \label{bb}
\partial_t u_n &\to \partial_t u \quad &\mbox{in} \quad &L^2({0, T; L^2}), \\ \label{cc}
u_n &\to u \quad &\mbox{in} \quad &L^{r + 1}(0, T; L^{r + 1}),
\end{align}
as $n \to \infty$. We note that $u \in L^2({0, T; H^2})$ and $\partial_tu \in L^2({0, T; L^2})$ implies that $u \in C([0, T]; H^1)$ (see e.g.\ Proposition $1.35$ in \cite{RRSbook}) and hence we also have
\begin{equation} \label{dd}
u_n \to u \quad \mbox{in} \quad C([0, T]; H^1).
\end{equation}

Since $v$ is a weak solution of the CBF equations, we have [note that $\dv{(u\ast\varphi_n)} = (\dv{u})\ast\varphi_n = 0$, so $u_n$ are valid test functions]
\begin{align} \nonumber
-\calkawoddoz{0}{t}{\dual{v}{\partial_t u_n}} &+ \calkawoddoz{0}{t}{\dual{\nabla v}{\nabla u_n}} + \calkawoddoz{0}{t}{\dual{(v \cdot \nabla)v}{u_n}} \\ \label{r:stronglemmawo}
&+ \calkawoddoz{0}{t}{\dual{\abs{v}^{r - 1}v}{u_n}} = \dual{v(0)}{u_n(0)} - \dual{v(t)}{u_n(t)}
\end{align}
for all $t \in [0, T]$. To prove the lemma, it is sufficient to pass to the limit in $\lr{\ref{r:stronglemmawo}}$.

Passing to the limit in the Navier--Stokes terms is standard and follows from $\lr{\ref{aa}}$, $\lr{\ref{bb}}$ and $\lr{\ref{dd}}$. Therefore, we can focus on the Brinkman--Frorchheimer nonlinearity; we note that by standard estimates and $\lr{\ref{cc}}$ we have
\begin{align} \nonumber
\abs{\calkawoddoz{0}{t}{\dual{\abs{v}^{r - 1}v}{u - u_n}}} &\leq \calkawoddoz{0}{t}{\int_{\T^3}{\abs{v}^r\abs{u - u_n}}} \leq \calkawoddoz{0}{t}{\norm{v}_{L^{r + 1}}^r\norm{u - u_n}_{L^{r + 1}}} \\ \nonumber
&\leq \norm{v}_{L^{r + 1}(0, T; L^{r + 1})}^r\norm{u - u_n}_{L^{r + 1}(0, T; L^{r + 1})} \to 0
\end{align}
as $n \to \infty$, which ends the proof.
\end{proof}

Finally, we can prove the main result of this section; we show that strong solutions are unique in the class of weak solutions satisfying the Energy Inequality (all Leray--Hopf weak solutions, not necessarily constructed via Galerkin approximation method). In the critical case when $r = 3$ (cubic nonlinearity $\abs{u}^2u$), since all weak solutions satisfy the Energy Equality (as shown in \cite{HR} on a periodic domain), this means that strong solutions are unique in the class of all weak solutions.

\begin{theorem}[Weak-strong uniqueness] \label{weakstrong}
Suppose that $u$ is a strong solution of the convective Brinkman--Forchheimer equations with $r \in \lra{1, 3}$ on the time interval $\lra{0, T}$, and that $v$ is any weak solution on $[0, T]$ arising from the same initial condition $v(0) = u(0) \in V^1$, that satisfies the Energy Inequality
\begin{align} \nonumber
\frac{1}{2}\norm{v(t)}^2 &+ \calkawoddoz{0}{t}{\norm{\nabla v(s)}^2} + \calkawoddoz{0}{t}{\norm{v(s)}^{r + 1}_{L^{r + 1}}} \leq \frac{1}{2}\norm{v(0)}^2
\end{align}
for $t \in \lra{0, T}$. Then $u \equiv v$ on $\lra{0, T}$.
\end{theorem}

\begin{proof}
From Lemmas $\ref{stronglemaone}$ and $\ref{stronglemmawo}$, we have for all $t \in \lra{0, T}$
\begin{align} \nonumber
\calkawoddoz{0}{t}{\dual{\partial_tu}{v}} &+ \calkawoddoz{0}{t}{\dual{\nabla u}{\nabla v}} + \calkawoddoz{0}{t}{\dual{(u \cdot \nabla)u}{v}} + \calkawoddoz{0}{t}{\dual{\abs{u}^{r - 1}u}{v}} = 0,
\\ \nonumber
-\calkawoddoz{0}{t}{\dual{v}{\partial_t u}} &+ \calkawoddoz{0}{t}{\dual{\nabla v}{\nabla u}} + \calkawoddoz{0}{t}{\dual{(v \cdot \nabla)v}{u}} + \calkawoddoz{0}{t}{\dual{\abs{v}^{r - 1}v}{u}} \\ \nonumber
&= \dual{v(0)}{u(0)} - \dual{v(t)}{u(t)}.
\end{align}
Adding the above equations, we obtain
\begin{align} \nonumber
2\calkawoddoz{0}{t}{\dual{\nabla u}{\nabla v}}+ \calkawoddoz{0}{t}{\dual{(u \cdot \nabla)u}{v}} + \calkawoddoz{0}{t}{\dual{(v \cdot \nabla)v}{u}} \\ \label{addedeqns}
+ \calkawoddoz{0}{t}{\dual{\abs{u}^{r - 1}u}{v}} + \calkawoddoz{0}{t}{\dual{\abs{v}^{r - 1}v}{u}} = \norm{u(0)}^2 - \dual{v(t)}{u(t)}.
\end{align}

Our goal now is to obtain an integral inequality for the difference of the solutions $w := v - u$. To this end, we use the following standard indentity
\begin{equation} \nonumber
\norm{a - b}^2 = \norm{a}^2 + \norm{b}^2 - 2\dual{a}{b}
\end{equation}
to deal with the linear terms. We get
\begin{align} \nonumber
2\dual{\nabla u}{\nabla v} &= \norm{\nabla u}^2 + \norm{\nabla v}^2 - \norm{\nabla w}^2,\\ \nonumber
\dual{v(t)}{u(t)} &= \frac{1}{2}\norm{u(t)}^2 + \frac{1}{2}\norm{v(t)}^2 - \frac{1}{2}\norm{w(t)}^2.
\end{align}
We use the relation $v = w + u$, and by standard properties of the convective term, we obtain
\begin{align} \nonumber
\dual{(u \cdot \nabla)u}{v} + \dual{(v \cdot \nabla)v}{u} = \dual{(w \cdot \nabla)w}{u}.
\end{align}
We use two different substitutions to deal with the absorption terms
\begin{align} \nonumber
\dual{\abs{u}^{r - 1}u}{v} + \dual{\abs{v}^{r - 1}v}{u} &= \dual{\abs{u}^{r - 1}u}{w + u} + \dual{\abs{v}^{r - 1}v}{v - w} \\ \nonumber
&= \norm{u}^{r + 1}_{L^{r + 1}} + \norm{v}^{r + 1}_{L^{r + 1}} - \dual{\abs{u}^{r - 1}u - \abs{v}^{r - 1}v}{w}.
\end{align}

Hence, it follows from $\lr{\ref{addedeqns}}$ that
\begin{align} \nonumber
-\calkawoddoz{0}{t}{\norm{\nabla w}^2} + \calkawoddoz{0}{t}{\norm{\nabla u}^2} + \calkawoddoz{0}{t}{\norm{\nabla v}^2} + \calkawoddoz{0}{t}{\dual{(w \cdot \nabla)w}{u}} \\ \nonumber
+ \calkawoddoz{0}{t}{\norm{u}^{r + 1}_{L^{r + 1}}} + \calkawoddoz{0}{t}{\norm{v}^{r + 1}_{L^{r + 1}}} - \calkawoddoz{0}{t}{\dual{\abs{u}^{r - 1}u - \abs{v}^{r - 1}v}{w}} \\ \nonumber
= \frac{1}{2}\norm{u(0)}^2 + \frac{1}{2}\norm{v(0)}^2 + \frac{1}{2}\norm{w(t)}^2 - \frac{1}{2}\norm{u(t)}^2 - \frac{1}{2}\norm{v(t)}^2.
\end{align}
Rearranging the terms in the above, we obtain the following equation for the~difference $w$
\begin{align} \label{weakstrongfinal}
\frac{1}{2}\norm{w(t)}^2 &+ \calkawoddoz{0}{t}{\norm{\nabla w}^2} - \calkawoddoz{0}{t}{\dual{(w \cdot \nabla)w}{u}} + I_1 = I_2 + I_3,
\end{align}
where
\begin{align} \nonumber
I_1 &:= \calkawoddoz{0}{t}{\dual{\abs{u}^{r - 1}u - \abs{v}^{r - 1}v}{u - v}} \geq 0, \\ \nonumber
I_2 &:= \frac{1}{2}\norm{u(t)}^2 + \calkawoddoz{0}{t}{\norm{\nabla u}^2} + \calkawoddoz{0}{t}{\norm{u}^{r + 1}_{L^{r + 1}}} - \frac{1}{2}\norm{u(0)}^2 = 0, \\ \nonumber
I_3 &:= \frac{1}{2}\norm{v(t)}^2 + \calkawoddoz{0}{t}{\norm{\nabla v}^2} + \calkawoddoz{0}{t}{\norm{v}^{r + 1}_{L^{r + 1}}} - \frac{1}{2}\norm{v(0)}^2 \leq 0.
\end{align}
We employed here the Energy Equality for the strong solution\footnote{The fact that strong solutions satisfy the Energy Equality is a simple consequence of Lemma \ref{r:stronglemaone}.} $u$ and the Energy Inequality for the weak solution $v$, and also the monotonicity of the absorption term (Lemma \ref{monotone}).

Therefore, we can estimate $\lr{\ref{weakstrongfinal}}$ in the following way
\begin{align} \nonumber
\frac{1}{2}\norm{w(t)}^2 + \calkawoddoz{0}{t}{\norm{\nabla w}^2} &\leq \abs{\calkawoddoz{0}{t}{\dual{(w \cdot \nabla)w}{u}}} \leq \calkawoddoz{0}{t}{\norm{u}_{L^{\infty}}\norm{w}\norm{\nabla w}} \\ \nonumber
&\leq \frac{1}{2}\calkawoddoz{0}{t}{\norm{u}_{L^{\infty}}^2\norm{w}^2} + \frac{1}{2}\calkawoddoz{0}{t}{\norm{\nabla w}^2} \\ \nonumber
&\leq c\calkawoddoz{0}{t}{\norm{u}_{H^2}^2\norm{w}^2} + \frac{1}{2}\calkawoddoz{0}{t}{\norm{\nabla w}^2};
\end{align}
we used the $3$D embedding $H^2 \hookrightarrow L^{\infty}$ in the last line. Then, we have
\begin{align} \nonumber
\norm{w(t)}^2 + \calkawoddoz{0}{t}{\norm{\nabla w}^2} \leq c\calkawoddoz{0}{t}{\norm{u}_{H^2}^2\norm{w}^2},
\end{align}
and consequently
\begin{align} \nonumber
\norm{w(t)}^2 \leq c\calkawoddoz{0}{t}{\norm{u(s)}_{H^2}^2\norm{w(s)}^2}.
\end{align}
Since $u$ is a strong solution
$$ \calkawoddoz{0}{t}{\norm{u(s)}_{H^2}^2} < \infty \quad \mbox{for all} \quad t \in \lra{0, T}, $$
so application of the integral version of the Gronwall Lemma yields that $w(t) = 0$ for all $t \in [0, T]$.
\end{proof}

As a straightforward corollary of Theorem \ref{weakstrong}, we can deduce a weaker result: uniqueness of strong solutions in the class of strong solutions.

\begin{corollary} \label{cbfsilnejednozn}
Let $u$ and $v$ be two strong solutions of the convective Brinkman--Forchheimer eqautions $\lr{\ref{cbfralphazero}}$ with $r \geq 0$ on the time interval $[0, T]$, starting from the same initial condition $u_0 \in V^1$. Then $u \equiv v$ for all times $t \le T$.
\end{corollary}

This result follows from Theorem \ref{weakstrong} only for the absorption exponents in the range $r \in \lra{1, 3}$; because strong solutions are by definition weak solutions with additional regularity and they satisfy the Energy Equality, it suffices to apply Theorem \ref{weakstrong} to the strong solutions $u$ and $v$. However, one can prove Corollary \ref{cbfsilnejednozn} independently for all exponents $r \geq 0$, following the proof for the analogous result for the Navier--Stokes equations; the only additional difficulty is in dealing with an extra nonlinear term $C_r(u)$. In this particular case, we can eliminate the additional nonlinearity from the proof due to its properties. We provide a short sketch of this fact below.

\begin{proof}
We set $w := u - v$. Then, of course $w(0) = 0$. We subtract weak formulations of the~CBF equations for the functions $u$ and $v$ and obtain equation for the difference
\begin{align} \nonumber
-\calkawoddoz{0}{t}{\dual{w}{\partial_t\varphi}} + \calkawoddoz{0}{t}{\dual{\nabla w}{\nabla\varphi}} &+ \calkawoddoz{0}{t}{\dual{B(u) - B(v)}{\varphi}} \\ \nonumber
&+ \calkawoddoz{0}{t}{\dual{C_r(u) - C_r(v)}{\varphi}} = -\dual{w(t)}{\varphi(t)}.
\end{align}
Using Lemma \ref{stronglemmawo}, we take as a test function $\varphi := w$ and get
\begin{align} \nonumber
\frac{1}{2}\norm{w(t)}^2 + \calkawoddoz{0}{t}{\norm{\nabla w}^2} &+ {\calkawoddoz{0}{t}{\dual{C_r(u) - C_r(v)}{w}}} \\ \label{uniqueineqcbf}
&\leq \abs{\calkawoddoz{0}{t}{\dual{B(u) - B(v)}{w}}}.
\end{align}

First, we deal with the nonlinearities connected with the operators $C_r$. We note that we can simply drop this term on the lef-hand side of $(\ref{uniqueineqcbf})$. Indeed, by monotonicity (Lemma \ref{monotone}), we have
\begin{align} \nonumber
\dual{C_r(u) - C_r(v)}{w} = \dual{\abs{u}^{r - 1}u - \abs{v}^{r - 1}v}{u - v} \geq 0 \quad \mbox{for} \quad r \geq 0.
\end{align}
Therefore, we can proceed as in the Navier--Stokes case to finish the proof of uniqueness of strong solutions.
\end{proof}

\section{Robustness of regularity} \label{s:robustness}

In this section, we deal with the so-called `\emph{robustness of regularity}' result for the~solutions of the convective Brinkman--Forchheimer equations on a torus $\mathbb{T}^3$. It generalises the result obtained in \cite{DashtiRobinson} for the Navier--Stokes equations.

We take $u_0$, $v_0 \in V^1$ and fix $T > T' > 0$. Let $u$ be a strong solution of the CBF equations on the time interval $[0, T]$ with external forces $f$ and initial condition $u_0$. Similarly, let $v$ be a strong solution of the CBF equations on $[0, T']$ with external forces $g$ and initial condition $v_0$. We will give an explicit condition, depending only on the data and on the function $u$, which allows us to extend (due to uniqueness) the function $v$ to a strong solution on the time interval $[0, T]$.

We consider the following system of equations
\begin{equation} \nonumber
\left\{\begin{split}
&\partial_tu + Au + B(u) + C_r(u) = f, \quad
&u(0, x) = u_0, \\
&\partial_tv + Av + B(v) + C_r(v) = g, \quad
&v(0, x) = v_0.
\end{split}\right.
\end{equation}
As before, we denote the difference of solutions by $w := u - v$. Subtracting the above equations we obtain the equation for $w$
\begin{equation} \label{differencecbf}
\partial_tw + Aw + B(u) - B(v) + C_r(u) - C_r(v) = f - g,
\end{equation}
with the initial condition
\begin{equation} \nonumber
w(0, x) = u_0 - v_0.
\end{equation}

\subsection{Technical ODE lemma}

In the proof of the robustness of regularity for the~above equations, the following simple ODE lemma will be extremely useful. It will allow us to estimate the time of existence for solutions of certain differential inequalities in terms of coefficients of a corresponding differential equation.

\begin{lemma} \label{l:odes}
Let $T > 0$, $a > 0$ and $n \in \mathbb{N}$ $(n > 1)$. Let $\delta(t)$ be a nonnegative, continuous function on the interval $[0, T]$. Let also $y$ be a nonnegative function, satisfying the following differential inequality
\begin{equation} \nonumber
\left\{\begin{split}
&\dot{y} \leq a y^n + \delta(t), \\
&y(0) = y_0 \geq 0.
\end{split}\right.
\end{equation}
We define the quantity
$$ \eta := y_0 + \int_0^T{\delta(t) \, \mathrm{d}t}. $$
If the following condition is satisfied
$$ \eta < \frac{1}{\left[(n - 1)a T\right]^{{1}/{(n - 1)}}}, $$
\begin{enumerate}
\item then $y(t)$ stays bounded on the interval $[0, T]$,

\item and \label{odes:jednostajnie} $y(t) \to 0$ as $\eta \to 0$, uniformly on $[0, T]$.
\end{enumerate}
\end{lemma}

For the proof of Lemma \ref{l:odes} see e.g.\ \cite{Constantin}, \cite{DashtiRobinson} or \cite{RRSbook}.

\subsection{A priori estimates}

Taking into account that $u$ is a strong solution on the~time interval $[0, T]$ and that we want to say the same about $v$, we need to change the~form of equation $\lr{\ref{differencecbf}}$ to eliminate the unknown function $v$. From the definition, we have $v = u - w$, so due to bilinearity of the form $B$, we have the identity
\begin{align} \nonumber
B(u) - B(v) &= B(u) - B(u - w) = B(u, u) - B(u - w, u- w) \\ \nonumber
&= B(u, u) - B(u, u) - B(w, w) + B(u, w) + B(w, u) \\ \nonumber
&= B(u, w) + B(w, u) - B(w, w).
\end{align}

Multiplying now both sides of $\lr{\ref{differencecbf}}$ by $Aw$ (we assume here that the function $w$ has sufficient regularity to justify these operations) and integrating over the spatial domain, we get
\begin{align} \nonumber
\dual{\partial_tw}{Aw} &+ \dual{Aw}{Aw} + \dual{B(u) - B(v)}{Aw} + \dual{C_r(u) - C_r(v)}{Aw} \\ \nonumber
&= \dual{f - g}{Aw}.
\end{align}

Integrating by parts, we obtain
\begin{align} \nonumber
\frac{1}{2}\frac{\mathrm{d}}{\mathrm{d}t}\norm{\nabla w}^2 + \norm{Aw}^2 &\leq \abs{\dual{f - g}{Aw}} + \abs{\dual{B(u, w) + B(w, u) - B(w, w)}{Aw}} \\ \label{nierjedenns}
&+ \abs{\dual{C_r(u) - C_r(v)}{Aw}}.
\end{align}

We will now estimate all the terms on the right-hand side of the inequality $\lr{\ref{nierjedenns}}$.
Using standard estimates for the bilinear form $B$ (cf.\ \cite{ConstantinFoias} or \cite{DashtiRobinson}) and Lemma $\ref{crucialestimate}$, we can estimate all the terms coming from the Navier--Stokes equations (see also Chapter $9.1$ in \cite{RRSbook}).

We have
\begin{enumerate}[{a) }]
\item
\begin{equation} \nonumber
\abs{\dual{f - g}{Aw}} \leq \dual{\abs{f - g}}{\abs{Aw}} \leq \norm{f - g}\norm{Aw} \leq c\norm{f - g}^2 + \frac{1}{16}\norm{Aw}^{2},
\end{equation}

\item
\begin{align} \nonumber
\abs{\dual{B(u, w)}{Aw}} &\leq \dual{\abs{u}\abs{\nabla w}}{\abs{Aw}} \leq \norm{u}_{L^6}\norm{\nabla w}_{L^3}\norm{Aw} \\ \nonumber
&\leq c\norm{u}_{H^1}\norm{\nabla w}^{{1}/{2}}\norm{\nabla w}_{L^6}^{{1}/{2}}\norm{Aw} \leq c\norm{u}_{H^1}\norm{w}_{H^1}^{{1}/{2}}\norm{Aw}^{{3}/{2}} \\ \nonumber
&\leq c\norm{u}_{H^1}^{4}\norm{w}_{H^1}^{2} + \frac{1}{16}\norm{Aw}^{2},
\end{align}

\item
\begin{align} \nonumber
\abs{\dual{B(w, u)}{Aw}} &\leq \dual{\abs{w}\abs{\nabla u}}{\abs{Aw}} \leq \norm{w}_{L^6}\norm{\nabla u}_{L^3}\norm{Aw} \\ \nonumber
&\leq c\norm{w}_{H^1}\norm{\nabla u}^{{1}/{2}}\norm{\nabla u}_{L^6}^{{1}/{2}}\norm{Aw} \leq c\norm{w}_{H^1}\norm{\nabla u}^{{1}/{2}}\norm{Au}^{{1}/{2}}\norm{Aw} \\ \nonumber
&\leq c\norm{w}_{H^1}^{2}\norm{\nabla u}\norm{Au} + \frac{1}{16}\norm{Aw}^{2},
\end{align}

\item
\begin{align} \nonumber
\abs{\dual{- B(w, w)}{Aw}} &\leq \dual{\abs{w}\abs{\nabla w}}{\abs{Aw}} \leq \norm{w}_{L^6}\norm{\nabla w}_{L^3}\norm{Aw} \\ \nonumber
&\leq c\norm{w}_{H^1}\norm{\nabla w}^{{1}/{2}}\norm{\nabla w}_{L^6}^{{1}/{2}}\norm{Aw} \leq c\norm{w}_{H^1}^{{3}/{2}}\norm{Aw}^{{3}/{2}} \\ \nonumber
&\leq c\norm{w}_{H^1}^{6} + \frac{1}{16}\norm{Aw}^{2}.
\end{align}
\end{enumerate}

Summing parts $\mathrm{a)}$ to $\mathrm{d)}$, we get
\begin{align} \nonumber
\abs{\dual{f - g}{Aw}} &+ \abs{\dual{B(u, w) + B(w, u) - B(w, w)}{Aw}} \leq c\norm{f - g}^2 \\ \label{szacjeden}
&+ c\lr{\norm{u}_{H^1}^{4} + \norm{\nabla u}\norm{Au}}\norm{w}_{H^1}^{2} + c\norm{w}_{H^1}^{6} + \frac{1}{4}\norm{Aw}^{2}.
\end{align}

Note that we obtain the full $H^1$-norm of the difference $w$ on the right-hand side of $(\ref{szacjeden})$ and there is only $L^2$-norm of the gradient on the left-hand side of $(\ref{nierjedenns})$. To circumvent that problem we can consider the energy equality for the difference
\begin{align} \nonumber
\frac{1}{2}\frac{\mathrm{d}}{\mathrm{d}t}\norm{w}^2 + \norm{\nabla w}^2 + \dual{B(u) - B(v)}{w} &+ {\dual{C_r(u) - C_r(v)}{w}} \\ \label{energydifference}
&= {\dual{f - g}{w}}.
\end{align}

We note again that ${\dual{C_r(u) - C_r(v)}{w}} \ge 0$. Substituting in $(\ref{energydifference})$ $v = u - w$, we get
\begin{align} \nonumber
\dual{B(u) - B(v)}{w} &= \dual{B(u, u) - B(u - w, u - w)}{w} = \dual{B(u, u)}{w} - \dual{B(u, u)}{w} \\ \nonumber
&+\dual{B(u, w)}{w} + \dual{B(w, u)}{w} - \dual{B(w, w)}{w} \\ \nonumber
&= \dual{B(w, u)}{w}.
\end{align}
We used here the fact that
$$ \dual{(u \cdot \nabla)v}{w} = -\dual{(u \cdot \nabla)w}{v} $$
for $u \in V^1$ and $v, w \in H^1$.

Therefore, we obtain from the energy equality $(\ref{energydifference})$
\begin{align} \nonumber
\frac{1}{2}\frac{\mathrm{d}}{\mathrm{d}t}\norm{w}^2 + \norm{\nabla w}^2 \le \abs{{\dual{f - g}{w}}} + \abs{\dual{B(w, u)}{w}}.
\end{align}
Estimating the nonlinear term gives
\begin{align} \nonumber
\abs{\dual{B(w, u)}{w}} &\leq \dual{\abs{w}\abs{\nabla u}}{\abs{w}} \leq \norm{w}^2_{L^4}\norm{\nabla u} \le \norm{w}^{1/2}\norm{w}^{3/2}_{L^6}\norm{\nabla u} \\ \nonumber
&\leq c\norm{w}^2_{H^1}\norm{\nabla u},
\end{align}
from which we conclude
\begin{align} \label{szacextra}
\frac{1}{2}\frac{\mathrm{d}}{\mathrm{d}t}\norm{w}^2 + \norm{\nabla w}^2 &\le c\norm{f - g}^2 + c\norm{w}^2_{H^1}\lr{\norm{\nabla u} + 1}.
\end{align}

To estimate the additional nonlinear terms in $(\ref{nierjedenns})$ connected with the operator $C_r$ we use Lemma \ref{cbfdifference}
$$ \abs{C_r(u) - C_r(v)} \leq \lr{2^{r - 2}r}\lr{\abs{u}^{r - 1}\abs{w} + \abs{w}^{r}} \quad \mbox{for} \quad r \geq 1, $$
which gives
\begin{equation}\label{dodajzero}
\abs{\dual{C_r(u) - C_r(v)}{Aw}} \leq \lr{2^{r - 2}r}\lra{\dual{\abs{u}^{r - 1}\abs{w}}{\abs{Aw}} + \dual{\abs{w}^r}{\abs{Aw}}}.
\end{equation}

We can estimate the first term in $\lr{\ref{dodajzero}}$ using H\"older's inequality with three exponents ${6}/({r - 1}), {6}/({4 - r}), 2$ and Sobolev's embedding $H^1 \hookrightarrow L^6$
\begin{align} \nonumber
\dual{\abs{u}^{r - 1}\abs{w}}{\abs{Aw}} &\leq \norm{u}_{L^{6}}^{r - 1}\norm{w}_{L^{{6}/({4 - r})}}\norm{Aw} \\ \nonumber
&\leq c\norm{u}_{H^1}^{r - 1}\norm{w}_{H^1}\norm{Aw}  \\ \label{cbfszaca}
&\leq c\norm{u}_{H^1}^{2\lr{r - 1}}\norm{w}_{H^1}^2 + \frac{1}{8}\norm{Aw}^2.
\end{align}

Using the same bound for $L^{2r}$-norm as in $(\ref{tworestimate})$, we estimate the second term on the right-hand side of $\lr{\ref{dodajzero}}$
\begin{align} \nonumber
\dual{\abs{w}^r}{\abs{Aw}} &\leq \norm{w}_{L^{2r}}^{r}\norm{Aw} \leq c\norm{w}_{L^{2r}}^{2r} + \frac{1}{8}\norm{Aw}^2 \\ \label{cbfszacb}
&\leq c\norm{w}_{H^1}^{2r} + \frac{1}{8}\norm{Aw}^2.
\end{align}

Combining the inequalities $\lr{\ref{cbfszaca}}$-$\lr{\ref{cbfszacb}}$ yields
\begin{align} \label{nieliniowosccbfdwa}
\abs{\dual{C_r(u) - C_r(v)}{Aw}} &\leq c\norm{u}_{H^1}^{2\lr{r - 1}}\norm{w}_{H^1}^2 + c\norm{w}_{H^1}^{2r} + \frac{1}{4}\norm{Aw}^2.
\end{align}
We apply $(\ref{szacjeden})$ and $(\ref{nieliniowosccbfdwa})$ in $(\ref{nierjedenns})$ and obtain
\begin{align} \nonumber
\frac{\mathrm{d}}{\mathrm{d}t}\norm{\nabla w}^2 + \norm{Aw}^2 &\leq c_0\norm{f - g}^2 + c_1\lr{\norm{u}_{H^1}^{4} + \norm{\nabla u}\norm{Au} + \norm{u}_{H^1}^{2\lr{r - 1}}}\norm{w}_{H^1}^{2} \\ \label{robinequalityA}
&+ c_2\norm{w}_{H^1}^{2r} + c_3\norm{w}_{H^1}^{6}.
\end{align}
Finally, we add together $(\ref{szacextra})$ and $\ref{robinequalityA}$
\begin{align} \nonumber
\frac{\mathrm{d}}{\mathrm{d}t}\norm{w}_{H^1}^2 + \norm{\nabla w}^2 &+ \norm{Aw}^2 \leq c_0\norm{f - g}^2 + c_2\norm{w}_{H^1}^{2r} + c_3\norm{w}_{H^1}^{6} \\ \label{robinequality}
&+c_1\lr{\norm{u}_{H^1}^{4} + \norm{\nabla u}\norm{Au} + \norm{u}_{H^1}^{2\lr{r - 1}} + \norm{\nabla u} + 1}\norm{w}_{H^1}^{2}.
\end{align}

\subsection{Robustness of regularity}

In this section, we prove the following theorem for the convective Brinkman--Forchheimer equations with $r \in [1, 3]$ on a periodic domain $\T^3$.
\begin{theorem} \label{twr:stabilitycbf}
Assume that $f, g \in L^2(0, T; H)$ and $u_0, v_0 \in V^1$. Furthermore, let $u \in L^{\infty}(0, T; V^1) \cap L^{2}(0, T; V^2)$ be the strong solution of the convective Brinkman--Forchheimer equations $\lr{\ref{cbfralphazero}}$ on the time interval $[0, T]$, with external forces $f$ and initial condition $u_0$. If
\begin{align} \label{stabilitycbf}
&\norm{u_0 - v_0}_{H^1}^2 + c_0\calkat{\norm{f(t) - g(t)}^2} < R(u),
\end{align}
where
$$ R(u) := c\frac{\exp\lr{-c_2T}}{\sqrt{T}}\exp{\lr{-c_1\calkat{\lr{\norm{u}_{H^1}^{4} + \norm{\nabla u}\norm{Au} + \norm{u}_{H^1}^{2\lr{r - 1}} + \norm{\nabla u}}}}}, $$
for some positive constants $c_0, c_1, c_2, c$, then the function $v$ solving the CBF equations $\lr{\ref{cbfralphazero}}$, with external forces $g$ and initial condition $v_0$, is also a strong solution on the time interval $[0, T]$ and have the same regularity as the function $u$.
\end{theorem}

The proof of the above theorem is similar to the proof of an analogous result for the Navier--Stokes equations (see \cite{DashtiRobinson} for the details).
\begin{proof}
Local existence of strong solutions for the CBF equations $\lr{\mbox{Theorem \ref{localcbf}}}$ implies that there exists $\tilde{T} > 0$ such that $v \in L^{\infty}(0, T'; V^1) \cap L^{2}(0, T'; V^2)$ for every $T' < \tilde{T}$. We denote the maximal time of existence of the strong solution $v$ by $\tilde{T}$, i.e.\
\begin{equation} \nonumber
\limsup_{t \to \tilde{T}^{-}}{\norm{\nabla v(t)}} = \infty.
\end{equation}
This implies that $\norm{\nabla w(t)}$ also blows up as $t \to \tilde{T}^-$, where $w := u - v$. We assume that $\tilde{T} \le T$, where $T$ is the time of existence of the strong solution $u$, and lead to a contradiction.

The difference $w$ satisfies
\begin{equation} \label{governcbf}
\partial_tw + Aw + B(u, w) + B(w, u) - B(w, w) + C_r(u) - C_r(v) = f - g
\end{equation}
on the interval $(0, \tilde{T})$, with the initial condition $w(0, x) = u_0 - v_0$. We know that $\partial_tv \in L^{2}(0, T'; H)$ for every $T' < \tilde{T}$. Furthermore, we have $\tilde{T} \le T$, so obviously also $\partial_tu \in L^{2}(0, T'; H)$ for every $T' < \tilde{T}$. Then, taking the inner product of $\lr{\ref{governcbf}}$ with $Aw$ in $L^2$ and using our a priori estimate $(\ref{robinequality})$, we obtain
\begin{align} \nonumber
\frac{\mathrm{d}}{\mathrm{d}t}\norm{w}_{H^1}^2 + \norm{Aw}^2 &\leq c_0\norm{f - g}^2 + c_r\norm{w}_{H^1}^{2r} + c_3\norm{w}_{H^1}^6 \\ \label{nierdwacbf}
&+ c_1\norm{w}_{H^1}^{2}\lr{\norm{u}_{H^1}^{4} + \norm{\nabla u}\norm{Au} + \norm{u}_{H^1}^{2\lr{r - 1}} + \norm{\nabla u} + 1},
\end{align}
for appriopriate values of the constants $c_i$, $i \in \lrb{0, 1, r, 3}$.

We define the quantities
\begin{itemize}
\item $X(t) := \norm{w(t)}_{H^1}^2$,

\item $\delta(t) := c_0\norm{f(t) - g(t)}^2$,

\item $\tilde{\gamma}(t) := c_1\lr{\norm{u(t)}_{H^1}^{4} + \norm{\nabla u(t)}\norm{Au(t)} + \norm{u(t)}_{H^1}^{2\lr{r - 1}} + \norm{\nabla u(t)} + 1}$.
\end{itemize}
Inequality $\lr{\ref{nierdwacbf}}$ gives (omitting $\norm{Aw}^2$ on the left-hand side)
\begin{align} \nonumber
X' &\leq c_3X^3 + c_rX^r + \tilde{\gamma}(t)X + \delta(t).
\end{align}
Using an inequality (valid for $X \ge 0$)
$$ X^p \leq X^3 + X \quad \mbox{for} \quad p \in [1, 3], $$
and changing the constant $c_3$, we get
\begin{align} \label{rrzjedencbf}
X' &\leq c_3X^3 + \gamma(t)X + \delta(t),
\end{align}
where $\gamma(t) := \tilde{\gamma}(t) + c_r$.

We now take
$$ Y(t) := \exp\lr{-\int_{0}^{t} \gamma(s) \, \mathrm{d}s}X(t) $$
and multiply both sides of $\lr{\ref{rrzjedencbf}}$ by $\exp\lr{-\int_{0}^{t} \gamma(s) \, \mbox{ds}} \leq 1$. This way we obtain
\begin{align} \nonumber
Y' &\leq c_3\exp\lr{-\int_{0}^{t} \gamma(s) \, \mathrm{d}s}X^3 + \delta(t)\exp\lr{-\int_{0}^{t} \gamma(s) \, \mathrm{d}s} \\ \nonumber
&\leq c_3\left[\exp\lr{2\int_{0}^{t} \gamma(s) \, \mathrm{d}s}\right]Y^3 + \delta(t) \\ \nonumber
&\leq \underbrace{c_3\left[\exp\lr{2\int_{0}^{T} \gamma(s) \, \mathrm{d}s}\right]}_{=: K}Y^3 + \delta(t).
\end{align}

Hence, we have the differential inequality [valid on the time interval $(0, \tilde{T})$]
\begin{equation} \nonumber
Y' \leq KY^3 + \delta(t),
\end{equation}
with the initial condition
\begin{equation} \nonumber
Y(0) = \norm{u_0 - v_0}_{H^1}^{2}.
\end{equation}

Therefore, by Lemma \ref{l:odes} used for $n = 3$, the function ${Y}(t)$ is uniformly bounded on the time interval $[0, T']$ for every $T' < \tilde{T} \le T$, provided that
$$ Y(0) + \int_0^{T'}{\delta(t) \, \mathrm{d}t} < \frac{1}{\lr{2KT'}^{1/2}}, $$
which clearly holds (since $T' < T$) if we have
$$ Y(0) + \int_0^{T}{\delta(t) \, \mathrm{d}t} < \frac{1}{\lr{2KT}^{1/2}}. $$
Substituting all our original variables in the above condition, we obtain
\begin{align} \nonumber
&\norm{u_0 - v_0}_{H^1}^2 + c_0\calkat{\norm{f(t) - g(t)}^2} \\ \nonumber
&< \frac{\exp{\lr{-c_rT}}}{\sqrt{2c_3T}}\exp{\lr{-c_1\calkat{\lr{\norm{u}_{H^1}^{4} + \norm{\nabla u}\norm{Au} + \norm{u}_{H^1}^{2\lr{r - 1}} + \norm{\nabla u} + 1}}}},
\end{align}
which is (up to a change of constants) the robustness condition $(\ref{stabilitycbf})$. If this condition is satisfied, it follows that the function $X(t) = \norm{w(t)}^2_{H^1}$ is uniformly bounded on the time interval $[0, \tilde{T})$
$$ X(t) = Y(t)\exp\lr{\int_{0}^{t} \gamma(s) \, \mathrm{d}s} \le Y(t)\exp\lr{\int_{0}^{T} \gamma(s) \, \mathrm{d}s} \le C(T) < \infty. $$
Hence, we finally get that $ \norm{w({{t}})}_{H^1} \le C(T)$ for all $t <\tilde{T}$, and consequently
$\norm{v(t)}_{H^1} \le C(T)$ for $t \in [0, \tilde{T})$ as well. It follows that
$$ \limsup_{t \to \tilde{T}^{-}}{\norm{\nabla v(t)}} \le C(T) < \infty, $$
which contradicts the maximality of the time $\tilde{T}$. Therefore, $\tilde{T} > T$ and the function $\nabla v(t)$ does not blow up, at least on the time interval $[0, T]$. Hence, the function $v$ belongs to the space $L^{\infty}(0, T; V^1)$.

Now, directly from the inequality $\lr{\ref{nierdwacbf}}$, it follows that the function $v({t})$ belongs also to the space $L^2(0, T; V^2)$, which proves that it is a strong solution on the time interval $[0, T]$, completing the proof of Theorem $\ref{twr:stabilitycbf}$.
\end{proof}

It is worth mentioning that the robustness of regularity result proved here could also be obtained via the Implicit Function Theorem.
Indeed, let us consider the map
$$
F\colon V^1\times L^2(0,T;H)\times L^2(0,T;V^1)\to V^1\times L^2(0,T;H)
$$
defined by
$$ F(u_0,f;u) := \lr{u(0)-u_0, \ \ \partial_tu - \mu Au + B(u) + \beta C_r(u) - \PP f}. $$
To apply the Implicit Function Theorem one first has to check that this map is $C^1$-smooth in all variables. Then, if $\bar u$ is a strong solution, the condition `$D_uF(\bar u)$ is invertible' is equivalent to the unique solvability of the linear problem
$$ \partial_tv - \mu\Delta v + (v \cdot \nabla)\bar u + (\bar u \cdot \nabla)v + \beta r\abs{\bar{u}}^{r - 1}v = h(t), \ \ \nabla\cdot v = 0,\ v|_{t=0}=v_0 $$
for all $h\in L^2(0,T;H)$ and all $v_0\in V^1$. Since the solution $\bar u$ is strong, this can be shown using energy estimates very similar to the analysis presented above, and the Implicit Function Theorem then gives existence of strong solutions for $(u_0,f)$ close to their original values.

While this avoids some of the arguments we used, the advantage of our approach is that it yields an estimate of the `robustness radius', while the Implicit Function Theorem only gives robustness for `sufficiently small' perturbations.

\section{Conclusion} \label{s:conclusion}

Going back to Theorem $\ref{twr:stabilitycbf}$, it is natural to ask what kind of condition, if any, is required if we consider `robustness of regularity' with respect to the absorption exponent $r$. To focus our attention on the dependence on the exponents, let us take $u_0 \equiv v_0 \in V^1$ and $f \equiv g \in L^2(0, T; H)$, in such a way that $u$ is a strong solution on the time interval $[0, T]$ of the CBF equations with initial condition $u_0$ and the exponent $s$ (if $s > 3$ we know that it is in fact global-in-time strong solution), and let $v$ be a weak solution of the CBF equations with initial condition $v_0$ and the exponent $r \in [1, 3]$, where $r < s$. We know that $v$ is also a strong solution on some time interval $[0, \tilde{T}]$. We want to find a condition for exponents $r$ and $s$ depending only on the function $u$, which ensures that $v$ remains strong at least on the time interval $[0, T]$.

The only new obstacle in the problem described above lays in estimating the difference $C_s(u) - C_r(v)$. We observe that
\begin{align} \nonumber
\abs{\abs{u}^{s - 1}u - \abs{v}^{r - 1}v} &\leq \abs{\abs{u}^{s - 1}u - \abs{u}^{r - 1}u} + \abs{\abs{u}^{r - 1}u - \abs{v}^{r - 1}v} \\ \label{differenceexp}
&\leq \abs{u}^r\abs{\abs{u}^{s - r} - 1} + \abs{\abs{u}^{r - 1}u - \abs{v}^{r - 1}v}.
\end{align}
We have already seen how to deal with the second term on the right-hand side of $(\ref{differenceexp})$ [cf.\ $(\ref{dodajzero})$ and the following lines]. Therefore, using similar arguments to those in the proof of Theorem $\ref{twr:stabilitycbf}$, we obtain the robustness condition for the absorption exponents
\begin{align} \label{stabilityexponent}
c_0\calkat{ \lr{ \int_{\T^3}{ \abs{u}^{2r}\abs{ \abs{u}^{s - r} - 1 }^2 \, \mathrm{d}x } }^{1/2}} < R(u, r),
\end{align}
where $R(u, r)$ is equal to the constant $R(u)$ defined in Theorem $\ref{twr:stabilitycbf}$; this constant is finite because $u$ is the strong solution of the CBF equations with the absorption exponent $s$ and $s > r$.
On the other hand, the term on the left-hand side of $(\ref{stabilityexponent})$ tends to $0$ as $s - r \to 0^+$ (provided that the integral is bounded). Fixing $r = 3$ and letting $s \to r^+$ we can see from the condition $(\ref{stabilityexponent})$ how close we have to get with $s$ to the critical case $r =3$ in order to ensure that the weak solution $v$ is actually a strong solution on the time interval $[0, T]$.

In the works of Chernyshenko et al.\ \cite{CCRT} and Dashti $\&$ Robinson \cite{DashtiRobinson} the robustness of regularity for the Navier--Stokes equations was used to construct a numerical algorithm which can verify in a finite time regularity of a given strong solution. The second ingredient required in that construction is the convergence of the Galerkin approximations to the strong solution. As we showed in this article, the robustness of regularity can be extended to the convective Brinkman--Forchheimer equations with the absorption exponent $r \in [1, 3]$. Using similar methods as presented here to deal with the additional nonlinearity $\abs{u}^{r - 1}u$, it should be possible to prove also for the CBF equations that the Galerkin approximations of a strong solution converge strongly to that solution in appropriate function spaces. Consequently, it should be possible to construct a similar algorithm for the numerical verification of regularity for these equations as well.

\bibliography{biblio}
\bibliographystyle{acm}
\end{document}